\documentclass[12pt]{article}
\usepackage{latexsym,amsfonts,amsthm,amsmath,amscd,amssymb}
\usepackage[dvips]{graphicx}

\newtheorem{lemma}{Lemma}[section]
\newtheorem{thm}[lemma]{Theorem}

\newtheorem{prop}[lemma]{Proposition}

\newcommand{\finedimo}{{\hfill\hbox{$\square$}\vspace{2pt}}}

\newcommand\matH{{\mathbb{H}}}
\newcommand\matE{{\mathbb{E}}}

\newcommand\matR{{\mathbb{R}}}

\renewcommand{\hbar}{{\overline{h}}}

\newfont{\Got}{eufm10 scaled 1200}

\newcommand{\vtil}[1]{\widetilde{v}\left(#1\right)}
\newcommand{\Mtil}{\widetilde{M}}

\newcommand{\diff}{\,\textrm{d}}

\newcommand\calA{{\mathcal A}}

\begin{document}

\title{Notes on the peripheral volume\\ of hyperbolic $3$-manifolds}

\author{Carlo~\textsc{Petronio}\and Michele~\textsc{Tocchet}}

\maketitle

\begin{abstract}
\noindent We consider hyperbolic 3-manifolds with either non-empty compact geodesic boundary,
or some toric cusps, or both. For any such $M$ we analyze what portion of the volume of $M$
can be recovered by inserting in $M$ boundary collars and cusp neighbourhoods with disjoint embedded interiors.
Our main result is that this portion can only be maximal in some combinatorially extremal configurations.
The techniques we employ are \emph{very} elementary but the result is in our opinion of some interest.
\end{abstract}

The issue of understanding volumes of hyperbolic $3$-manifolds has
been one of the central themes of geometric topology since the pioneering
work of Jorgensen and Thurston, who showed that the set of possible volumes is a well-ordered
subset of $\matR$ (see,~\emph{e.g.},~\cite{BePe}).
In particular, considerable energy has been devoted to
identifying the minima of the volume within given classes of manifolds,
and the following instances have by now been settled:
\begin{itemize}
\item The minimal-volume compact hyperbolic $3$-manifolds with non-empty geodesic
boundary were proved~\cite{KoMi} to be the
 $8$ manifolds that decompose into two truncated regular tetrahedra of
dihedral angle $\frac{\pi}6$ (these $8$ manifolds were first
described in~\cite{Fujii});
\item The minimal-volume hyperbolic $3$-manifolds with toric cusps (but no boundary)
were shown~\cite{CaoMe} to be the figure-eight knot complement and its sibling
(these manifolds were first described in~\cite{bibbia} and~\cite{CaHiWe});
\item The minimal-volume closed hyperbolic $3$-manifold was very recently proved~\cite{GaMeMi1, GaMeMi2, GaMeMi3, Mi}
to be the Weeks manifold
(first described in~\cite{CaHiWe} and~\cite{MaFo});
\item The two minimal-volume hyperbolic $3$-manifolds with two toric cusps (but no boundary)
were identified in~\cite{Agol}.
\end{itemize}
Other instances, however, remain open, as that of manifolds with one toric cusp
and one compact boundary component.

The general (and very roughly described)
idea of the papers quoted so far is to show that an upper or lower
bound on the volume implies constraints on the topology of a manifold.
Moreover, estimates on the volume of boundary collars or cusp neighbourhoods often
play an important role.
We also mention that the same ideas appear in other interesting articles,
such as~\cite{BriKa} and~\cite{DeBloSha}.
In this paper we address the following
natural question:
\begin{itemize}
\item \emph{Let $M$ be a hyperbolic $3$-dimensional manifold with either non-empty compact geodesic boundary,
or some toric cusps, or both. What is the optimal way of inserting in $M$
boundary collars and/or cusp neighbourhoods having disjoint embedded interiors?}
\end{itemize}
Optimality here is of course meant in the sense of volume maximization,
collars are defined using the distance, and cusp neighbourhoods are bounded
by horospherical cross-sections.
We will refer to a boundary collar or cusp neighbourhood using the collective term of
\emph{peripheral component}. Employing \emph{very} elementary techniques we will prove in this
paper that an optimal choice of the peripheral components necessarily occurs in
a combinatorially extremal configuration, a fact that we consider to be of some interest.

To state our main results we denote by $v$ the volume function and
we establish the convention that, when we mention distinct peripheral components, they
always have disjoint and embedded interiors. 
We then have:

\begin{prop}\label{2:periph:prop}
Let $M$ have two peripheral components $P_1$ and $P_2$. Then $v\left(P_1\right)+v\left(P_2\right)$ can have a local maximum
only if, up to reordering the indices, $P_1$ is chosen so that $v\left(P_1\right)$ is
maximal regardless of $P_2$, and then $P_2$ is chosen so that $v\left(P_2\right)$ is maximal
given $P_1$.
\end{prop}

(When both peripheral components are cusps, or both are collars of boundary components having the
same genus, we will also show that
one must maximize first the component that individually can be made bigger than the other one.)

For the next result we will need to refer to a certain modified volume
$\vtil{P}$ of a peripheral component $P$, that coincides with $v\left(P\right)$ when
$P$ is a cusp neighbourhood, and will be defined below when $P$ is a boundary collar.

\begin{thm}\label{3:periph:thm}
Let $M$ have three peripheral components $P_1$, $P_2$, and $P_3$. Then $v\left(P_1\right)+v\left(P_2\right)+v\left(P_3\right)$ can have a local maximum
only in one of the following configurations:
\begin{itemize}
\item Up to reordering the indices, first $P_1$ is chosen so that $v\left(P_1\right)$ is
maximal regardless of $P_2$ and $P_3$, next $P_2$ is chosen so that $v\left(P_2\right)$ is maximal
given $P_1$, and last $P_3$ is chosen so that $v\left(P_3\right)$ is maximal
given $P_1,P_2$;
\item Each of $P_1,P_2,P_3$ is tangent to the other two, and
the modified volumes $\vtil{P_1},\vtil{P_2},\vtil{P_3}$ satisfy the strict triangular inequalities.
\end{itemize}
Moreover, if in the latter case no $P_j$ is individually maximal, then
the configuration indeed gives a local maximum for $v\left(P_1\right)+v\left(P_2\right)+v\left(P_3\right)$.
\end{thm}

Proposition~\ref{2:periph:prop} and Theorem~\ref{3:periph:thm} of course have similar flavours, since they
both state that a local maximum of the peripheral volume is attained at a combinatorially non-generic
configuration, namely one involving
individual maximality of one peripheral component or a cycle of tangencies.
A punctual analysis similar to that given by Theorem~\ref{3:periph:thm} is perhaps possible
also for four or more peripheral components, but we will refrain from carrying it out. We confine
ourselves to the following general result (a more refined version of which will be stated and proved below):

\begin{prop}\label{general:prop}
Let $M$ have $n$ peripheral components $P_1,\ldots,P_n$ such that:
\begin{itemize}
\item No $P_j$ is such that $v(P_j)$ is maximal regardless of the other $P_i$'s;
\item There are fewer than $n$ tangencies between different $P_j$'s.
\end{itemize}
Then the configuration cannot be a local maximum
for $v\left(P_1\right)+\ldots+v\left(P_n\right)$.
\end{prop}

We conclude this introduction with two remarks. First, one could simplify the attempt to partially recover the volume of $M$
by imposing all the collars of the boundary components to have the same width, and all
the toric cusps to have the same volume, but our results show that this attempt
would typically be very inefficient. Second, we explain why we do not consider
in this paper the case of manifolds having annular cusps, the reason being that
for the boundary components entering such cusps one could not take a collar at all, so the theory
would have limited interest.

\section{Preliminaries}
Let us fix for the rest of this article a hyperbolic $3$-manifold $M$ with
either non-empty compact geodesic boundary,
or some toric cusps, or both. We will denote by $\Sigma_1,\ldots,\Sigma_b$
the components of $\partial M$ and by $T_1,\ldots,T_k$ the toric boundary components of
the compactification of $M$. Moreover for $d>0$ we will indicate by $U_j\left(d\right)$ the $d$-collar
of $\Sigma_j$ in $M$, and for $v>0$ by $C_i\left(v\right)$ the volume-$v$ horospherical cusp neighbourhood at $T_i$ in $M$.
When certain $U_j\left(d_j\right)$'s and $C_i\left(v_i\right)$'s are simultaneously considered we will always assume
that they have embedded and disjoint interiors. We begin by recalling that each $\Sigma_j$ has
a well-defined hyperbolic area $\calA\left(\Sigma_j\right)=-2\pi\chi\left(\Sigma_j\right)$. We then prove the following:

\begin{lemma}\label{area:T:lem}
Suppose that $M$ has both geodesic boundary and cusps.
Then the cusp torus $T_i$ has a well-defined area $\calA_j\left(T_i\right)=\frac{\calA\left(\matE^2/\Lambda_i\right)}{r_j^2}$ relative to the boundary component $\Sigma_j$,
where:
\begin{itemize}
\item A universal cover $p:\Mtil\to M$ is chosen with $\Mtil$ being
the intersection of a family of hyperbolic half-spaces in the half-space
model $\matE^2\times\left(0,+\infty\right)$ of $\matH^3$, in such a way that
with respect to $p$ the torus $T_i$ lifts to $\infty$;
\item $\Lambda_i$ is the lattice acting horizontally on $\matH^3$ to give the $i$-th cusp of $M$;
\item $r_j$ is the maximal Euclidean radius of a half-sphere centered at $\matE^2\times\{0\}$
that bounds $\Mtil$ and projects in $M$ to $\Sigma_j$.
\end{itemize}
\end{lemma}

\begin{proof}
A universal cover as in the statement exists and $r_j$ is well-defined because the maximum has to be taken
over a $\Lambda_i$-equivariant family of half-spheres. Two distinct universal covers as in the statement
differ by the composition of a horizontal Euclidean isometry and a dilation in $\matE^2\times\left(0,+\infty\right)$, and
the ratio in the statement is preserved by both.
\end{proof}

We now start dealing with peripheral volume, by quoting the following formula from~\cite{GaMeMi1}:

\begin{prop}\label{U:vol:prop}\quad
$\displaystyle{v\left(U_j\left(d\right)\right)=\frac{\calA\left(\Sigma_j\right)}4\cdot\left(2d+\sinh\left(2d\right)\right).}$
\end{prop}

Before proceeding, we define the modified volume $\widetilde{v}$ for a boundary collar
$U_j\left(d\right)$ by setting
$$\widetilde{v}\left(U_j\left(d\right)\right)=\frac{\calA\left(\Sigma_j\right)}4\cdot\left(1+\cosh\left(2d\right)\right).$$
Note that $\widetilde{v}\left(U_j\left(d\right)\right)$ is a strictly increasing function of $v\left(U_j\left(d\right)\right)$, but one cannot
describe the function explicitly.

\bigskip

The next result will be the core of our arguments. It
describes how
a peripheral component changes when it varies subject
to the condition of staying tangent to another one that is also varying.

\begin{prop}\label{variation:prop}\ \\

\vspace{-.8cm}

\begin{itemize}
\item If $C_{i_1}\left(v_1\right)$ and $C_{i_2}\left(v_2\right)$ vary while remaining tangent to each other,
then the product $v_1\cdot v_2$ remains constant;
\item If $U_j\left(d\right)$ and $C_i\left(v\right)$ vary while remaining tangent to each other
then $v=\frac{\calA_j\left(T_i\right)}2\cdot e^{-2d}$;
\item If $U_{j_1}\left(d_1\right)$ and $U_{j_2}\left(d_2\right)$
vary while remaining tangent to each other,
then the sum $d_1+d_2$ remains constant.
\end{itemize}
\end{prop}

\begin{proof}
Let us prove the first statement. Suppose that in some universal cover contained in
$\matH^3=\matE^2\times\left(0,+\infty\right)$ the cusp $C_{i_1}\left(v_1\right)$ is the quotient of the
horoball $\matE^2\times\left[z_{i_1},+\infty\right)$ acted on horizontally by a lattice $\Lambda_{i_1}$.
If the cusp changes so that in $M$ its boundary moves of some small distance $d\in\matR$
(with $d<0$ meaning that the cusp is shrinking) then it becomes the quotient under $\Lambda_{i_1}$
of $\matE^2\times\left[z,+\infty\right)$ with
$$\int\limits_z^{z_{i_1}}\frac1t\diff t=d$$
whence $z=z_{i_1}e^{-d}$. The cusp volume $v_1$ then changes from
$$\calA\left(\matE^2/\Lambda_{i_1}\right)\cdot\int\limits_{z_{i_1}}^{+\infty}\frac1{t^3}\diff t=
\frac{\calA\left(\matE^2/\Lambda_{i_1}\right)}2\cdot z_{i_1}^{-2}$$
to $\frac{\calA\left(\matE^2/\Lambda_{i_1}\right)}2 \cdot z^{-2}$, namely it changes by a factor
$e^{2d}$. During a simultaneous variation of $v_1$ and $v_2$ with
$C_{i_1}\left(v_1\right)$ and $C_{i_2}\left(v_2\right)$ remaining tangent to each other, the boundary of
$C_{i_2}\left(v_2\right)$ moves by a distance $-d$. The calculations already carried out show that then its volume
varies by a factor $e^{-2d}$, and the conclusion follows.

Turning to the second statement, we choose a universal cover of $M$ with $C_i\left(v\right)$ lifting
to some horoball $\matE^2\times \left[z,+\infty\right)$ in $\matE^2\times \left(0,+\infty\right)$.
With $r_j$ being as in the definition of $\calA_j\left(T_i\right)$, the condition that
$U_j\left(d\right)$ is tangent to $C_i\left(V\right)$ implies that
$$\int\limits_{r_j}^z\frac1t\diff t=d$$
whence $z=r_je^d$ and
$$v=\frac{\calA\left(\matE^2/\Lambda_{i_1}\right)}{2r_j^2}\cdot e^{-2d}=
\frac{\calA_j\left(T_i\right)}2\cdot e^{-2d}.$$
This proves the desired formula.
The third statement is obvious.
\end{proof}

\section{Two peripheral components\\ and the general case}

We will establish here the two propositions stated in the introduction.
For the sake of conciseness let us say that a
a peripheral component $P$ is \emph{maximal}
if $v(P)$ is maximal regardless of the other peripheral components,
\emph{i.e.}, if $\partial P$ is tangent either to itself or to $\partial M$.

\bigskip

\noindent\emph{Proof of Proposition}~\ref{2:periph:prop}.\quad
We must show that $v\left(P_1\right)+v\left(P_2\right)$ cannot have a local maximum unless, up
to switching indices, $P_1$ is maximal and $P_2$
is either maximal or tangent to $P_1$. If one of $P_1$ or $P_2$ is not maximal or tangent to the other one, of course $v\left(P_1\right)+v\left(P_2\right)$ cannot be locally
maximal. We are then left to exclude only the situation where
$P_1$ and $P_2$ are tangent to each other but neither of them is tangent to itself or to the boundary.
To this configuration we can locally apply Proposition~\ref{variation:prop}.
Depending on whether $P_1$ and $P_2$ are both cusps, a cusp and a collar,
or two collars, the total volume $V=v\left(P_1\right)+v\left(P_2\right)$
with respect to an appropriate parameter is given by
$$\begin{array}{lcl}
C_1\left(v_1\right)\cup C_2\left(v_2\right): & & V\left(v_1\right)=v_1+\frac {v_0^2}{v_1}\quad \hbox{for some}\ v_0>0, \\ \\
U_1\left(d_1\right)\cup C_1\left(v_1\right): & & V\left(d_1\right)=\frac{\calA\left(\Sigma_1\right)}4\cdot\left(2d_1+\sinh\left(2d_1\right)\right)+
\frac{\calA_1\left(T_1\right)}2\cdot e^{-2d_1},\\ \\
U_1\left(d_1\right)\cup U_2\left(d_2\right): & & V\left(d_1\right)=\frac{\calA\left(\Sigma_1\right)}4\cdot\left(2d_1+\sinh\left(2d_1\right)\right)\\
& & \phantom{V\left(d_1\right)}+\frac{\calA\left(\Sigma_2\right)}4\cdot\left(2\left(2d_0-d_1\right)+\sinh\left(2\left(2d_0-d_1\right)\right)\right)\\
& & \phantom{xxxxxxxxxxxxxxxxxxxxxxxxx}\quad \hbox{for some}\ d_0>0.
\end{array}$$
Since in all cases $V$ is a convex function and our starting point is in the interior of the domain of definition of $V$,
the conclusion follows.
\finedimo

In two special cases for the types of the two peripheral components we have the following
improvement on Proposition~\ref{2:periph:prop}, also announced above:

\begin{prop}\label{2:cusps:prop}
Suppose that $M$ has either two cusps and no geodesic boundary or two geodesic boundary components
of the same genus and no cusps. For $j=1,2$ let $v_j^{\max}$ be the maximal volume that
the $j$-th peripheral component $P_j$ can attain regardless of the other one.
If $v_1^{\max}\geqslant v_2^{\max}$ then the maximum of $v(P_1)+v(P_2)$ is attained
by maximizing first $P_1$ and then $P_2$ given $P_1$.
\end{prop}

\begin{proof}
We start with the case of two cusps.
Taking $v_1$ as a variable to parameterize $C_1\left(v_1\right)\cup C_2\left(v_2\right)$ in tangency position,
we have $v_2=\frac{v_0^2}{v_1}$ for some $v_0>0$, and $v_1$ varies in $\left[v_1^{\min},v_1^{\max}\right]$
with $v_1^{\min}=\frac {v_0^2}{v_2^{\max}}$. We must maximize on $\left[v_1^{\min},v_1^{\max}\right]$
the convex function $v_1\mapsto v_1+\frac {v_0^2}{v_1}$, that attains its minimum at $v_1={v_0}$.
We may now have $v_1^{\min}>{v_0}$ or $v_1^{\min}\leqslant{v_0}$.
In the former case of course the maximum of $v_1+\frac{v_0^2}{v_1}$
is attained at $v_1^{\max}$. In the latter case we have $v_2^{\max}=\frac{v_0^2}{v_1^{\min}}\geqslant {v_0}$,
but $v\mapsto v+\frac {v_0^2}v$ is increasing on $\left[{v_0},+\infty\right)$,
therefore the inequalities $v_1^{\max}\geqslant v_2^{\max}\geqslant {v_0}$ imply
$$v_1^{\max}+\frac{v_0^2}{v_1^{\max}}\geqslant v_2^{\max}+\frac {v_0^2}{v_2^{\max}}=\frac {v_0^2}{v_1^{\min}}+v_1^{\min}$$
and the conclusion follows.

\medskip

For two boundary components of the same genus $g$ the argument is similar.
Let $d_j^{\max}$ be the maximal $d_j$ such that $U_j(d_j)$ has embedded interior,
and let $d_i^{\min}$ for $\{i,j\}=\{1,2\}$ be the maximal
$d_i$ such that $U_i(d_i)$ has embedded interior in $M\setminus U_j\left(d_j^{\max}\right)$.
Note that $d_1^{\max}+d_2^{\min}=d_1^{\min}+d_2^{\max}$, and denote by $2d_0$ this quantity.
Since $v_j^{\max}=\pi(g-1)\left(2d_j^{\max}+\sinh\left(2d_j^{\max}\right)\right)$ we have
$d_1^{\max}\geqslant d_2^{\max}$. Using $d_1$ as a deformation parameter we must
now maximize on $\left[d_1^{\min},d_1^{\max}\right]$ the convex function
$$f(d_1)=\pi(g-1)\left(2d_1+\sinh\left(2d_1\right)+
2\left(2d_0-d_1\right)+\sinh\left(2\left(2d_0-d_1\right)\right)\right),$$
that attains its minimum at $d_0$. If $d_1^{\min}\geqslant d_0$ the maximum
is of course attained at $d_1^{\max}$. If $d_1^{\min}<d_0$
we have $d_2^{\max}>d_0$, but we also know that
$d_1^{\max}\geqslant d_2^{\max}$, whence
$$f\left(d_1^{\max}\right)\geqslant f\left(d_2^{\max}\right)=f\left(d_1^{\min}\right)$$
and the proof is complete.
\end{proof}

We conclude this section by proving a result that easily implies Proposition~\ref{general:prop}:

\begin{prop}
Let $M$ have $n$ peripheral components $P_1,\ldots,P_n$
and construct a graph with vertices $P_1,\ldots,P_n$ and edges joining
peripheral components that are tangent to each other. Suppose that this graph has a connected
component $\Gamma$ such that:
\begin{itemize}
\item $\Gamma$ is a tree;
\item No vertex $P_j$ of $\Gamma$ is maximal.
\end{itemize}
Then the configuration cannot be a local maximum
for $v\left(P_1\right)+\ldots+v\left(P_n\right)$.
\end{prop}

\begin{proof}
Under the stated assumptions we can locally deform the peripheral components corresponding
to the vertices in $\Gamma$ using one parameter that can be both increased and decreased.
If in $\Gamma$ there are only cusps and no
boundary components, choosing $v_1$ as a deformation parameter we see
that each other $v_i$ in $\Gamma$ varies either as $c_i\cdot v_1$ or as $\frac{c_i}{v_1}$ for some
$c_i>0$, which implies that the sum of all the volumes of the peripheral components in $\Gamma$
is a convex function of $v_1$, whence the conclusion.

Suppose then that in $\Gamma$ there exists at least one boundary collar component and choose $d_1$ as a deformation parameter.
We then claim that each $d_j$ in $\Gamma$ varies as either $c_j+d_1$ or as $c_j-d_1$ for some $c_j\in\matR$,
and each $v_i$ in $\Gamma$ varies as either $c_i\cdot e^{-2d_1}$ or as $c_i\cdot e^{2d_1}$
for some $c_i>0$.
This can be easily checked using Proposition~\ref{variation:prop} and induction on the number
of edges in $\Gamma$ one needs to travel through in passing from $U_1\left(d_1\right)$ to $U_j\left(d_j\right)$ or $C_i\left(v_i\right)$.
Since $v\left(U_j\left(d_j\right)\right)$ is a convex function of $d_j$ we deduce that the sum of the volumes of the peripheral components
in $\Gamma$ is a convex function of $d_1$, and the proposition is proved.
\end{proof}

\paragraph{Experimental facts}
A census was carried out in~\cite{FriMaPe} of all the $5{,}192$ hyperbolic $3$-manifolds
with non-empty compact geodesic boundary that can be triangulated using up to $4$ tetrahedra.
It turns out that the geodesic boundary is always connected, and that there is one cusp for 31
manifolds and two cusps for one manifold. We were now able to check~\cite{tesi} that for all the 32
relevant cases the largest peripheral volume is obtained by maximizing first
the boundary collar and then the cusp neighbourhood(s, that both become tangent to the boundary
collar when there are two of them, before becoming tangent themselves or to each other).
).

\section{Three peripheral components}

This section is devoted to the proof of Theorem~\ref{3:periph:thm}.
We already know from Proposition~\ref{general:prop} that at a local maximum
for $v\left(P_1\right)+v\left(P_2\right)+v\left(P_3\right)$ either some $P_j$ is maximal or each $P_j$
is tangent to each other $P_i$. In the former case, suppose that
$P_1$ is maximal. If $P_2$ or $P_3$ is maximal given $P_1$
then up to switching $P_2$ and $P_3$
we have a configuration as described in the first item of the statement.
Otherwise $P_2$ and $P_3$ are
tangent to each other but not to themselves or to $P_1$, and the argument given in the proof of
Proposition~\ref{2:periph:prop} shows that the configuration cannot
locally maximize the volume.

We are left to deal with the configuration in which each $P_j$ is tangent to
each other $P_{i}$ but it is not individually maximal.
This implies that each $P_j$ is not tangent to itself and that it has positive width if it is of collar type
(otherwise some other $P_i$ would be maximal).
In this case the local deformation we can perform
is as follows:
\begin{itemize}
\item We can inflate $P_1$ and shrink $P_2$ and $P_3$ so that they stay tangent to $P_1$,
with shapes varying as described in Proposition~\ref{variation:prop};
\item We can shrink $P_1$, in which case we can further deform $P_2$ and $P_3$
in such a way that they stay tangent to each other.  But thanks to the argument
showing Proposition~\ref{2:periph:prop} we know that along this deformation
we cannot have a local maximum except at the extrema, namely when either $P_2$ or $P_3$ is tangent
to $P_1$.
\end{itemize}
To analyze exactly how the volume behaves under this deformation we need to distinguish according
to the types of $P_1,P_2,P_3$, noting that our choice of $P_1$ in the above description of
the deformation was an arbitrary one. The argument is similar in all four cases, with complications
increasing (and amount of details we supply decreasing)
as the number of boundary components grows.

\medskip\noindent\textsc{Case I: Three cusps}\quad
Let $C_1\left(v_1^{\left(0\right)}\right),C_2\left(v_2^{\left(0\right)}\right),C_3\left(v_3^{\left(0\right)}\right)$
be the initial cusps with indices chosen so that
$v_2^{\left(0\right)}\geqslant v_3^{\left(0\right)}$. We then let $v_1$ vary in a neighbourhood of $v_1^{\left(0\right)}$
and note that, according to the above description of the deformation, for $v_1\leqslant v_1^{\left(0\right)}$
the total deformed volume is given by $v_1$ plus
$$\max\left\{\frac{v_1^{\left(0\right)}\cdot v_2^{\left(0\right)}}{v_1}+\frac{v_3^{\left(0\right)}}{v_1^{\left(0\right)}}\cdot  v_1,\
\frac{v_2^{\left(0\right)}}{v_1^{\left(0\right)}}\cdot  v_1 +\frac{v_1^{\left(0\right)}\cdot v_3^{\left(0\right)}}{v_1}\right\},$$
but the assumption $v_2^{\left(0\right)}\geqslant v_3^{\left(0\right)}$ readily implies that the maximum is given by
the first expression. Therefore near $v_1^{\left(0\right)}$ the
total deformed volume is
$$V\left(v_1\right)=
\left\{\begin{array}{lll}
\displaystyle{v_1+\frac{v_1^{\left(0\right)}\cdot v_2^{\left(0\right)}}{v_1}+\frac{v_3^{\left(0\right)}}{v_1^{\left(0\right)}}\cdot  v_1} & \textrm{for} & v_1\leqslant v_1^{\left(0\right)}\\
& & \\
\displaystyle{v_1+\frac{v_1^{\left(0\right)}\cdot v_2^{\left(0\right)}}{v_1}+\frac{v_1^{\left(0\right)}\cdot v_3^{\left(0\right)}}{v_1}} & \textrm{for} & v_1\geqslant v_1^{\left(0\right)}.
\end{array}
\right.$$
We now have
$$V'_-\left(v_1^{\left(0\right)}\right) = 1-\frac{v_2^{\left(0\right)}}{v_1^{\left(0\right)}}+\frac{v_3^{\left(0\right)}}{v_1^{\left(0\right)}},\qquad
V'_+\left(v_1^{\left(0\right)}\right) = 1-\frac{v_2^{\left(0\right)}}{v_1^{\left(0\right)}}-\frac{v_3^{\left(0\right)}}{v_1^{\left(0\right)}}$$
and we note that under the assumption
$v_2^{\left(0\right)}\geqslant v_3^{\left(0\right)}$
the strict triangular inequalities for $v_1^{\left(0\right)},v_2^{\left(0\right)},v_3^{\left(0\right)}$ read as
$$v_2^{\left(0\right)}-v_3^{\left(0\right)} < v_1^{\left(0\right)} < v_2^{\left(0\right)}+v_3^{\left(0\right)},$$
therefore they are equivalent to the conditions
$$V'_+\left(v_1^{\left(0\right)}\right)<0<V'_-\left(v_1^{\left(0\right)}\right),$$
so they imply that $V$ has a local maximum at $v_1^{\left(0\right)}$. We must show that, conversely,
if one of the strict triangular inequalities is violated then $V$ does not have a local maximum at
$v_1^{\left(0\right)}$, which is because in this case we have
either $V'_-\left(v_1^{\left(0\right)}\right)\leqslant 0$ or $V'_+\left(v_1^{\left(0\right)}\right)\geqslant 0$, but
$V''_-\left(v_1^{\left(0\right)}\right)>0$ and $V''_+\left(v_1^{\left(0\right)}\right)>0$.

\medskip\noindent\textsc{Case II: One boundary component and two cusps}\quad
We denote by $\Sigma$ the boundary component and for $i=2,3$ by $\calA\left(T_i\right)$ the areas relative
to $\Sigma$ of the tori $T_2$ and $T_3$, with indices chosen so that $\calA\left(T_2\right)\geqslant\calA\left(T_3\right)$.
Suppose that the initial peripheral components
are $U_1\left(d_1^{\left(0\right)}\right),C_2\left(v^{\left(0\right)}_2\right),C_3\left(v_3^{\left(0\right)}\right)$, with volumes
$$v_1^{\left(0\right)}=\frac{\calA\left(\Sigma\right)}4\cdot \left(2d_1^{\left(0\right)}+\sinh\left(2d_1^{\left(0\right)}\right)\right),\qquad
v_i^{\left(0\right)}=\frac{\calA\left(T_i\right)}2\cdot e^{-2d_1^{\left(0\right)}}$$
(therefore $v_2^{(0)}\geqslant v_3^{(0)}$ by our choice of the indices).
Using $d_1$ to parameterize the deformation we have for
$d_1<d_1^{\left(0\right)}$ that the deformed total volume is given by
$\frac{\calA\left(\Sigma\right)}4\cdot \left(2d_1+\sinh\left(2d_1\right)\right)$ plus
$$\begin{array}{l}
\displaystyle{\max\left\{\frac{\calA\left(T_2\right)}2\cdot e^{-2d_1}+
\frac{\calA\left(T_3\right)}2\cdot e^{-2\left(2d_1^{\left(0\right)}-d_1\right)},\right.}\\
\displaystyle{\left.\phantom{xxxxxxxxxxxxxxxxxx}
\frac{\calA\left(T_2\right)}2\cdot e^{-2\left(2d_1^{\left(0\right)}-d_1\right)}+\frac{\calA\left(T_3\right)}2\cdot e^{-2d_1}\right\}}
\end{array}$$
and the first expression prevails thanks to the assumption $\calA\left(T_2\right)\geqslant\calA\left(T_3\right)$.
The deformed total volume is therefore
$$V\left(d_1\right)=\left\{\begin{array}{l}
\displaystyle{\frac{\calA\left(\Sigma\right)}4\cdot \left(2d_1+\sinh\left(2d_1\right)\right)+
\frac{\calA\left(T_2\right)}2\cdot e^{-2d_1}+\frac{\calA\left(T_3\right)}2\cdot e^{2d_1-4d_1^{\left(0\right)}}} \\
\phantom{xxxxxxxxxxxxxxxxxxxxxxxxxxxxxxxxxxxx}
\textrm{for}\quad d_1\leqslant d_1^{\left(0\right)}\\ \\
\displaystyle{\frac{\calA\left(\Sigma\right)}4\cdot \left(2d_1+\sinh\left(2d_1\right)\right)+
\frac{\calA\left(T_2\right)}2\cdot e^{-2d_1}+\frac{\calA\left(T_3\right)}2\cdot e^{-2d_1}} \\
\phantom{xxxxxxxxxxxxxxxxxxxxxxxxxxxxxxxxxxxx}
\textrm{for}\quad d_1\geqslant d_1^{\left(0\right)}
    \end{array}\right.$$
and we easily have
$$V'_-\left(d_1^{\left(0\right)}\right)=2\left(\widetilde{v}_1^{\left(0\right)}-v_2^{\left(0\right)}+v_3^{\left(0\right)}\right),\qquad
V'_+\left(d_1^{\left(0\right)}\right)=2\left(\widetilde{v}_1^{\left(0\right)}-v^{\left(0\right)}_2-v^{\left(0\right)}_3\right).$$
Since $V''_{\pm}\left(d_1^{\left(0\right)}\right) > 0$, the conclusion follows precisely as in Case I.

\medskip\noindent\textsc{Case III: Two boundary components and one cusp}\quad
Let the peripheral components be $C_1\left(v_1^{\left(0\right)}\right),
U_2\left(d_2^{\left(0\right)}\right),U_3\left(d_3^{\left(0\right)}\right)$, whence
$$v_1^{\left(0\right)}=\frac{\calA_{2}\left(T_1\right)}2\cdot e^{-2d_2^{\left(0\right)}}
=\frac{\calA_{3}\left(T_1\right)}2\cdot e^{-2d_3^{\left(0\right)}}.$$
For the sake of brevity we now set $f_j\left(t\right)=\frac{\calA\left(\Sigma_j\right)}4\left(2t+\sinh\left(2t\right)\right)$,
so that $v\left(U_j\left(d_j\right)\right)=f_j\left(d_j\right)$, and $f'_j\left(d_j^{\left(0\right)}\right)=2\widetilde{v}^{\left(0\right)}_j$.
We now choose indices so that $\widetilde{v}^{\left(0\right)}_2\geqslant \widetilde{v}^{\left(0\right)}_3$.
We deform the configuration using the parameter $v_1$, to do which we define the functions
$$d_j\left(v_1\right)=-\frac12\log\frac{2v_1}{\calA_{j}\left(T_1\right)},$$
noting that for $v_1\geqslant v_1^{\left(0\right)}$ the peripheral configuration is given by
$$C_1\left(v_1\right),U_2\left(d_2\left(v_1\right)\right),U_3\left(d_3\left(v_1\right)\right).$$ For $v_1<v_1^{\left(0\right)}$, on the contrary,
we have one of the following:
\begin{itemize}
\item $d_2=d_2\left(v_1\right)$ and $d_3=d_3^{\left(0\right)}+d_2^{\left(0\right)}-d_2\left(v_1\right)$,
\item $d_3=d_3\left(v_1\right)$ and $d_2=d_2^{\left(0\right)}+d_3^{\left(0\right)}-d_3\left(v_1\right)$.
\end{itemize}
The total
deformed volume for $v_1\leqslant v_1^{\left(0\right)}$ is then $v_1$ plus
$$\begin{array}{l}
\displaystyle{\max\Big\{
f_2\left(d_2\left(v_1\right)\right)+f_3\left(d_3^{\left(0\right)}+d_2^{\left(0\right)}-d_2\left(v_1\right)\right),}\\
\displaystyle{\phantom{xxxxxxxxxxxxxxxxxx}
f_2\left(d_2^{\left(0\right)}+d_3^{\left(0\right)}-d_3\left(v_1\right)\right)+f_3\left(d_3\left(v_1\right)\right)\Big\}}
\end{array}$$
and the first expression prevails thanks to our assumption, because at the point $v_1 = v_1^{(0)}$
it has the same value as the second expression but smaller first derivative.
Therefore
$$V\left(v_1\right)=
\left\{\begin{array}{lll}
v_1+f_2\left(d_2\left(v_1\right)\right)+f_3\left(d_3^{\left(0\right)}+d_2^{\left(0\right)}-d_2\left(v_1\right)\right) & \textrm{for} & v_1\leqslant v_1^{\left(0\right)}\\ & & \\
v_1+f_2\left(d_2\left(v_1\right)\right)+f_3\left(d_3\left(v_1\right)\right) & \textrm{for} & v_1\geqslant v_1^{\left(0\right)}
\end{array}\right.$$
whence
$$V'_-\left(v_1^{\left(0\right)}\right)=1-\frac{\widetilde{v}^{\left(0\right)}_2}{v_1^{\left(0\right)}}+\frac{\widetilde{v}^{\left(0\right)}_3}{v_1^{\left(0\right)}},\qquad
V'_+\left(v_1^{\left(0\right)}\right)=1-\frac{\widetilde{v}^{\left(0\right)}_2}{v_1^{\left(0\right)}}-\frac{\widetilde{v}^{\left(0\right)}_3}{v_1^{\left(0\right)}}$$
and precisely as in the previous two cases we conclude that we have a local maximum if
and only if $v_1^{\left(0\right)},\widetilde{v}^{\left(0\right)}_2,\widetilde{v}^{\left(0\right)}_3$ satisfy the strict triangular inequalities.

\medskip\noindent\textsc{Case IV: Three boundary components}\quad
Let the boundary components be $\Sigma_1,\Sigma_2,\Sigma_3$ and the initial configuration be
$U_1\left(d_1^{\left(0\right)}\right),U_2\left(d_2^{\left(0\right)}\right),U_3\left(d_3^{\left(0\right)}\right)$ with
$\widetilde{v}^{\left(0\right)}_2\geqslant \widetilde{v}^{\left(0\right)}_3$.
Setting
$$f_j\left(t\right)=\frac{\calA\left(\Sigma_j\right)}4\left(2t+\sinh\left(2t\right)\right)$$
and using $d_1$ to parameterize the deformation
we have that the deformed volume is given by
$$V\left(d_1\right)=\left\{\begin{array}{l}
f_1\left(d_1\right)+f_2\left(d_2^{\left(0\right)}+d_1^{\left(0\right)}-d_1\right)+f_3\left(d_3^{\left(0\right)}-d_1^{\left(0\right)}+d_1\right)\\
\phantom{xxxxxxxxxxxxxxxxxxxxxxxxxxxxxxxxxxxx} \textrm{for}\quad d_1\leqslant d_1^{\left(0\right)}\\ \\
f_1\left(d_1\right)+f_2\left(d_2^{\left(0\right)}+d_1^{\left(0\right)}-d_1\right)+f_3\left(d_3^{\left(0\right)}+d_1^{\left(0\right)}-d_1\right)\\
\phantom{xxxxxxxxxxxxxxxxxxxxxxxxxxxxxxxxxxxx} \textrm{for}\quad d_1\geqslant d_1^{\left(0\right)}.
\end{array}\right.$$
This gives
$$V'_-\left(d_1^{\left(0\right)}\right)=
2\left(\widetilde{v}_1^{\left(0\right)}-\widetilde{v}_2^{\left(0\right)}+\widetilde{v}_3^{\left(0\right)}\right),\qquad
V'_+\left(d_1^{\left(0\right)}\right)=2\left(\widetilde{v}_1^{\left(0\right)}-\widetilde{v}_2^{\left(0\right)}-\widetilde{v}_3^{\left(0\right)}\right)$$
and the conclusion is once again the same.

\vspace{1cm}

\noindent
Dipartimento di Matematica Applicata\\
Universit\`a di Pisa\\
Via Filippo Buonarroti, 1C\\
56127 PISA -- Italy\\
petronio@dm.unipi.it\\

\noindent
Dipartimento di Matematica\\
``Sapienza'' Universit\`a di Roma\\
Piazzale Aldo Moro, 5\\
00185 ROMA -- Italy\\
tocchet@mat.uniroma1.it

\noindent

\end{document}